\documentclass[reqno]{amsart}
\usepackage[backref=page]{hyperref}
\usepackage{cite}
\usepackage{enumitem}
\usepackage{color,soul} 

\theoremstyle{plain}
\newtheorem{theorem}{Theorem}
\newtheorem{proposition}[theorem]{Proposition}

\newtheorem{lemma}[theorem]{Lemma}

\theoremstyle{definition}

\theoremstyle{remark}
\newtheorem{remark}{Remark}

\allowdisplaybreaks

\begin{document}
	
	\author{Phuong Le}
	\address{Phuong Le$^{1,2}$ (ORCID: 0000-0003-4724-7118)\newline
		$^1$Faculty of Economic Mathematics, University of Economics and Law, Ho Chi Minh City, Vietnam; \newline
		$^2$Vietnam National University, Ho Chi Minh City, Vietnam}
	\email{phuongl@uel.edu.vn}
	
	
	
	
	\title[Symmetry of bounded solutions]{Symmetry of bounded solutions to quasilinear elliptic equations in a half-space}
	\begin{abstract}
		Let $u$ be a bounded positive solution to the problem $-\Delta_p u = f(u)$ in $\mathbb{R}^N_+$ with zero Dirichlet boundary condition, where $p>1$ and $f$ is a locally Lipschitz continuous function. Among other things, we show that if $f(\sup_{\mathbb{R}^N_+} u)=0$ and $f$ satisfies some other mild conditions, then $u$ depends only on $x_N$ and monotone increasing in the $x_N$-direction. Our result partially extends a classical result of Berestycki, Caffarelli and Nirenberg in 1993 to the $p$-Laplacian.
	\end{abstract}
	
	\maketitle
	
	\section{Introduction}
	We study the monotonicity and one-dimensional symmetry of bounded solutions to the problem
	\begin{equation}\label{main}
		\begin{cases}
			-\Delta_p u = f(u) &\text{ in } \mathbb{R}^N_+,\\
			u > 0 &\text{ in } \mathbb{R}^N_+,\\
			u = 0 &\text{ on } \partial\mathbb{R}^N_+,
		\end{cases}
	\end{equation}
	where $\mathbb{R}^N_+ := \{x=(x',x_N)\in\mathbb{R}^{N-1}\times\mathbb{R} \mid x_N>0\}$ denotes the upper half-space, $f$ is a locally Lipschitz continuous function and $-\Delta_p u = -\text{div}(|\nabla u|^{p-2}\nabla u)$ is the $p$-Laplacian of $u$.
	
	When $p=2$, the problem reduces to
	\begin{equation}\label{semilinear}
		\begin{cases}
			-\Delta u = f(u) &\text{ in } \mathbb{R}^N_+,\\
			u > 0 &\text{ in } \mathbb{R}^N_+,\\
			u = 0 &\text{ on } \partial\mathbb{R}^N_+.
		\end{cases}
	\end{equation}
	Problems of this type arise in the derivation of a priori bounds for positive solutions of nonlinear second
	order elliptic problems on smooth bounded domains, in the study of semilinear problems with small diffusion on smooth bounded domains and in the study of regularity results for some free boundary
	problems (see \cite{MR619749,MR1044809,MR1260436,MR1470317,MR2145284}).
	
	The monotonicity of classical solutions to \eqref{semilinear} was studied in a series of seminal papers by Berestycki, Caffarelli and Nirenberg \cite{MR1260436,MR1395408,MR1655510,MR1470317}. Among other results, they showed in \cite{MR1395408,MR1655510} that if $f$ is a Lipschitz continuous function with $f(0) \ge0$ then any classical solution of \eqref{semilinear} is monotone increasing in the $x_N$-direction, furthermore, $\frac{\partial u}{\partial x_N} > 0$ in $\mathbb{R}^N_+$. When $f(0) \ge0$ and $f$ is merely locally Lipschitz continuous, the same monotonicity can be proved for positive solutions which are bounded in each strip $\{x\in\mathbb{R}^N \mid 0<x_N<\lambda\}$, see \cite{MR4142367}. The case $f(0) <0$ is more involved, and a complete proof of the monotonicity is only known in dimension $N=2$, see \cite{MR3593525,MR3641643}. These results were obtained via the well-known method of moving planes, which was introduced by Alexandrov \cite{MR143162} in the context of differential geometry and by Serrin \cite{MR333220} in the study of an overdetermined problem.
	
	The one-dimensional symmetry (also called rigidity) of solutions was also investigated by several authors. In particular, using the sliding method, Berestycki, Caffarelli and Nirenberg \cite{MR1470317} proved the following result (see also \cite{MR794096,MR937538}): Assume that $f:[0,+\infty)\to\mathbb{R}$ is a Lipschitz continuous function such that for some $0<t_0<t_1<\rho$,
	\begin{enumerate}
		\item[(i)] $f(t)>0$ on $(0,\rho)$ and $f(t)\le0$ on $[\rho,+\infty)$,
		\item[(ii)] $f(t)>c_0 t$ on $(0,t_0)$ for some $c_0>0$,
		\item[(iii)] $f$ is nonincreasing on $(t_1, \rho)$.
	\end{enumerate}
	Then every solution $u$ of \eqref{semilinear} depends only on $x_N$ and monotone increasing in the $x_N$-direction.
	For more general nonlinearities, it was shown in \cite{MR1260436} that if $u$ is a bounded positive solution of \eqref{semilinear} and $f(\sup_{\mathbb{R}^N_+}u)\le0$, then $u$ is one-dimensional.
	
	The monotonicity and symmetry of solutions to $p$-Laplace problems \eqref{main} were also studied by several authors. Since the $p$-Laplacian is nonlinear and singular (for $1<p<2$) or degenerate (for $p>2$), many properties which hold for \eqref{semilinear} do not hold for \eqref{main} in general, including the $C^2$ regularity, the weak and strong comparison principles. Hence it is not easy to extend the known results for problem \eqref{semilinear} to problem \eqref{main}. Nevertheless, some progress has been made in the study of qualitative properties of solutions to \eqref{main} in the last decades.	
	The monotonicity of solutions to \eqref{main} with bounded gradients was established by Farina, Montoro, Riey and Sciunzi via the moving plane method, see \cite{MR3303939,MR2886112} for the case $1 < p < 2$ and \cite{MR3752525,MR3118616}  for the case $p > 2$. A key assumption in these works is that $f:[0,+\infty)\to\mathbb{R}$ is a locally Lipschitz continuous function which is positive on $(0,+\infty)$. To the best of our knowledge, there are only two monotonicity results for \eqref{main} with sign-changing nonlinearities, which were obtained in \cite{MR4439897} and \cite{MR3371569}. The precise assumption in \cite{MR4439897} is that $\frac{2N+2}{N+2}<p<2$, $f:[0,+\infty)\to\mathbb{R}$ is locally Lipschitz continuous with $f(0)\ge0$ and the zeros set of $f$ is discrete. More recently, there are some results by the author in which $f$ is only locally Lipschitz continuous on $(0,+\infty)$, see \cite{10.1007/s11118-024-10157-1,2024arXiv240900365L}.
	
	When the monotonicity of solutions is in force, one may pass it to the one-dimensional symmetry via the De Giorgi type results which were available in low dimensions (see \cite{MR2483642} for instance). Following this approach, some rigidity results were obtained in \cite{MR2654242} for the case $N=2$, $p>\frac{3}{2}$ and \cite{MR3118616,MR2886112} for the case $N=3$, $p>\frac{8}{5}$ both under the boundedness of solutions and their gradients. For higher dimensions, we refer to Du and Guo \cite{MR2056284}, where the one-dimensional symmetry of bounded solutions to \eqref{main} is obtained under the assumption that
	\begin{enumerate}
		\item[(i)] $f(t)>0$ on $(0,\rho)$ and $f(t)<0$ on $(\rho,+\infty)$ for some $\rho>0$,
		\item[(ii)] $f(t)>c_0 t^{p-1}$ on $(0,t_0)$ for some $c_0,t_0>0$.
	\end{enumerate}
	This extends the main result in \cite{MR1470317} from the case $p=2$ to $p\ne2$.
	Notice that in all these works, the nonlinearity $f$ is positive in the range of solutions $u$.
	The true sign changing nonlinearity was considered recently by the author in \cite{MR4771444} but with the additional assumption that $u(x',x_N)\to1$ as $x_N\to+\infty$ uniformly in $x'\in\mathbb{R}^{N-1}$ and $\frac{2N+2}{N+2}<p\le2$. The method used in \cite{MR4771444} is the sliding method which relies heavily on the assumption of this uniform limit.
	
	In this paper, we will prove a symmetry result for problem \eqref{main} with possibly sign-changing nonlinearity and without assumptions on uniform limits in all dimensions and in the full range $1<p<+\infty$.
	Our result is therefore stronger and covers the one in \cite{MR4771444}.
	Throughout the present paper, we denote
	\[
	F(t)=\int_{0}^{t} f(s) ds
	\]
	for $t>0$. We also denote $Z_f:=\{z\in[0,+\infty) \mid f(z)=0\}$. By $B_r(x_0)$, we mean the open ball in $\mathbb{R}^N$ of center $x_0$ with radius $r>0$. Our main result is the following.
	
	\begin{theorem}\label{th:main}
		Assume $p>1$ and $f:[0,+\infty)\to\mathbb{R}$ is a continuous function which is locally Lipschitz continuous in $[0,+\infty)\setminus Z_f$. Let $u\in C^1(\overline{\mathbb{R}^N_+})$ be a bounded solution to \eqref{main}. Assume that $\rho=\sup_{\mathbb{R}^N_+} u$ verifies $f(\rho)=0$ and $f$ has no zeros in $(\rho-\varepsilon,\rho)$ for some $\varepsilon>0$. Moreover, assume that		
		\begin{equation}\label{smp_cond}
			\liminf_{t\to z^+} \frac{f(t)}{(t-z)^{p-1}} > -\infty,\quad
			\limsup_{t\to z^-} \frac{f(t)}{(z-t)^{p-1}} < +\infty
			\quad\text{ for all } z\in Z_f
		\end{equation}
		and one of the following conditions hold
		\begin{enumerate}
			\item[(i)] $f(0)\ge0$,
			\item[(ii)] $F(\rho)\ne0$,
			\item[(iii)] $u\in C^2(\mathbb{R}^N_+\cap B_r(x_0))$ for some $x_0\in\partial\mathbb{R}^N_+$ and $r>0$.
		\end{enumerate}
		Then $u$ depends only on $x_N$ and monotone increasing in $x_N$. Furthermore,
		\begin{equation}\label{deri}
			\frac{\partial u}{\partial x_N} > 0 \text{ in } \mathbb{R}^N_+.
		\end{equation}
	\end{theorem}
	
	\begin{remark}		
		Condition \eqref{smp_cond} ensures the strong maximum principle holds at the zeros of $f$. This condition is fulfilled by locally Lipschitz functions $f$ in the case $1<p\le2$. Indeed, in such a case,
		\[
		\frac{|f(t)|}{|t-z|^{p-1}} = \left|\frac{f(t)-f(z)}{t-z}\right| |t-z|^{2-p} \le C |t-z|^{2-p} \le C'
		\]
		for $t$ sufficiently close to $z \in Z_f$.
	\end{remark}
	
	\begin{remark}	
		Theorem \ref{th:main} partially extends a classical result in \cite{MR1260436} to the $p$-Laplacian case. More precisely, in \cite[Theorem 1]{MR1260436}, Berestycki, Caffarelli and Nirenberg proved the same result for $p=2$ without assuming that $f$ has no zeros in a left neighborhood of $\rho$. We need this assumption as well as (i), (ii) or (iii) to deal with the lack of the strong comparison principle and comparison boundary point lemma for the $p$-Laplacian with $p\ne2$.
		
		We also remark that (iii) holds in the case $p=2$ by standard elliptic regularity. Moreover, if $p\ne2$ and $\frac{\partial u}{\partial x_N}(x_0)>0$ for some $x_0\in\partial\mathbb{R}^N_+$, then $\frac{\partial u}{\partial x_N}>0$ in $\mathbb{R}^N_+\cap B_r(x_0)$ for some $r>0$ and (iii) is fulfilled thanks to standard regularity for nondegenerate elliptic operators.
	\end{remark}
	
	To prove Theorem \ref{th:main}, we first study the ODE problem
	\begin{equation}\label{1D}
		\begin{cases}
			-\Delta_p u = f(u) & \text{ in } \mathbb{R}_+,\\
			u \ge 0 & \text{ in } \mathbb{R}_+,\\
			u(0)=0, ~ \|u\|_{L^\infty(\mathbb{R}_+)}<+\infty,
		\end{cases}
	\end{equation}
	where $\mathbb{R}_+=(0,+\infty)$.
	One important ingredient of the proof of Theorem \ref{th:main} is the following classification result for \eqref{1D}, which is of independent interest.
	\begin{theorem}\label{th:1D}
		Assume $p>1$ and $f:[0,+\infty)\to\mathbb{R}$ is a continuous function which is locally Lipschitz continuous in $[0,+\infty)\setminus Z_f$ such that \eqref{smp_cond} holds. Then we have the following claims:
		\begin{enumerate}
			\item[(i)] For each $\rho\in Z_f^*:=\{z\in Z_f \mid F(t) < F(z) \text{ for all } t\in [0,z)\} \cup Z_f^0$, where			
			\[
			Z_f^0=
			\begin{cases}
				\{z\in Z_f \mid F(t) < 0 = F(z) \text{ for all } t\in (0,z)\} &\text{ if } f(0) < 0,\\
				\emptyset &\text{ otherwise},
			\end{cases}
			\]		
			there exists a unique solution $u_\rho$ to \eqref{1D} such that $u_\rho'>0$ in $\mathbb{R}_+$, $\lim\limits_{t\to+\infty} u_\rho(t)=\rho$ and $u_\rho'(0)=\left(\frac{p}{p-1}F(\rho)\right)^\frac{1}{p}$.
			\item[(ii)] If $f(0)<0$, then for
			\[
			\rho\in P_f := \{z>0 \mid F(t) < 0 = F(z) \text{ for all } t\in (0,z) \text{ and } f(z)>0\},
			\]
			there exists a unique solution $\tilde u_\rho$ to \eqref{1D} such that $\tilde u_\rho'>0$ in $(0,t^*)$, $\tilde u_\rho'(0)=\tilde u_\rho'(t^*)=0$, $\tilde u_\rho(t^*)=\rho$ and
			\[
			\tilde u_\rho(t+2kt^*)=\tilde u_\rho(t), \quad \tilde u_\rho(t+(2k+1)t^*)=\tilde u_\rho(t^*-t)
			\]
			for all $t\in[0,t^*]$ and all $k\in\mathbb{N}$,
			where
			\[
			t^* = \left(\frac{p-1}{p}\right)^\frac{1}{p}\int_{0}^{\rho} \frac{ds}{\left[F(\rho)-F(s)\right]^\frac{1}{p}}.
			\]
		\end{enumerate}
		
		Moreover, every solution $u \not\equiv 0$ to \eqref{1D} must have one of the above two forms.
	\end{theorem}
	
	\begin{remark}
		Dancer, Yu and Efendiev \cite{MR3058211} classified all solutions to \eqref{1D} under restriction $f(0)\ge0$. Here we drop that restriction in our classification result. Unlike the case $f(0)\ge0$, when $f(0)<0$, solutions $u$ with $u'(0)=0$ may exist and they may not strictly positive. This is due to the lack of the strong maximum principle and Hopf's lemma in the case $f(0)<0$. Moreover, one easily checks that $Z_f^0$ and $P_f$ each has at most one element. Hence Theorem \ref{th:1D} indicates that \eqref{1D} admits at most one positive solution $u$ with $u'(0)=0$ and at most one periodic solution and such solutions do not exist in the case $f(0)\ge0$.
		One well-known example in higher dimensions is given by the problem
		\[
		\begin{cases}
			-\Delta u = u-1 &\text{ in } \mathbb{R}^N_+,\\
			u = 0 &\text{ on } \partial\mathbb{R}^N_+,
		\end{cases}
		\]
		This problem has a unique nonnegative solution $u(x)=1-\cos x_N$, which is bounded, nonnegative, one-dimensional and periodic in $x_N$. This elegant result was conjectured by Berestycki, Caffarelli and Nirenberg \cite{MR1655510} and was proved in \cite{MR3593525} for dimension two and in \cite{MR3550849} for higher dimensions.
	\end{remark}
	
	As an application of Theorem \ref{th:main}, we have the following.
	\begin{theorem}\label{th:sub}
		Assume $p>1$. If $N>p+1$, we further assume $N(p-2)+2\ge0$. Let $f:[0,+\infty)\to\mathbb{R}$ be a locally Lipschitz continuous function such that
		\begin{enumerate}
			\item[(i)] for some $\rho>0$,
			\[
			f(t)>0 \text{ on } (0,\rho),
			\quad f(t)<0 \text{ on } (\rho,+\infty),
			\]
			\item[(ii)] $f$ satisfies \eqref{smp_cond} at $z=\rho$, i.e.,
			\[
			\liminf_{t\to \rho^+} \frac{f(t)}{(t-\rho)^{p-1}} > -\infty,\quad
			\limsup_{t\to \rho^-} \frac{f(t)}{(\rho-t)^{p-1}} < +\infty,
			\]
			\item[(iii)] there exist $c_0,t_0>0$ such that
			\[
			f(t) > c_0 t^\gamma \text{ for all } t \in (0, t_0),
			\]
		\end{enumerate}
		where
		\begin{align*}
			\gamma&=\frac{(p-1)(N-1)}{N-p-1} \text{ if } N>p+1,\\
			\gamma&\ge1 \text{ is arbitrary if } N\le p+1.
		\end{align*}
		If $p\ge2$, we further assume that
		\[
		\text{ either }\lim_{t\to0^+} \frac{f(t)}{t^{p-1}} = 0 \text{ or } \liminf_{t\to0^+} \frac{f(t)}{t^{p-1}} > 0.
		\] 
		Then any bounded solution $u$ of \eqref{main} depends only on $x_N$ and is monotone increasing in $x_N$. Furthermore,
		\[
		\frac{\partial u}{\partial x_N} > 0 \text{ in } \overline{\mathbb{R}^N_+}.
		\]
	\end{theorem}
	
	Due to assumption (i), parameter $\gamma$ in assumption (iii) must satisfy $\gamma\ge1$. Indeed, we can check that $\frac{(p-1)(N-1)}{N-p-1}\ge1$ if $N>p+1$ and $N(p-2)+2\ge0$.
	
	\begin{remark}
		Theorem \ref{th:sub} extends corresponding results in \cite{MR1470317,MR2002396} from $p=2$ to $p\ne2$. We also recall that a similar result was obtained by Du and Guo in \cite{MR2056284} for $p>1$ and $\gamma=p-1$. Du and Guo also conjectured that the same rigidity result is still valid if (iii) only holds for some $\gamma>p-1$. Hence Theorem \ref{th:sub} partially answers that open question left in \cite{MR2056284}.
	\end{remark}
	
	The rest of this paper is organized as follows. In Section \ref{sect2}, we recall some important analytic tools that will be used later. We prove Theorem \ref{th:1D} in Section \ref{sect3}. In Section \ref{sect4} we prove an existence result in balls. Section \ref{sect5} is devoted to the proofs of our main results, which are Theorems \ref{th:main} and \ref{th:sub}.
	
	\section{Preliminaries}\label{sect2}
	In this section, we recall some well-known tools which will be used later.
	The following strong maximum principle and Hopf's lemma are due to V\'{a}zquez \cite{MR768629}.
	\begin{theorem}[Strong maximum principle and Hopf's lemma]\label{th:smp}
		Let $u\in C^1(\Omega)$ be a nonnegative weak solution to
		\[
		-\Delta_p u + cu^q = g \ge 0 \quad\text{ in }\Omega,
		\]
		where $p>1$, $q\ge p-1$, $c\ge0$, $\Omega$ is a connected domain of $\mathbb{R}^N$ and $g \in L^\infty_{\rm loc}(\Omega)$. If $u \not\equiv 0$ then $u > 0$ in $\Omega$. Moreover for any point $x_0 \in \partial \Omega$ where the interior sphere condition is satisfied, and such that $u\in C^1(\Omega\cup\{x_0\})$ and $u(x_0) = 0$ we have that $\frac{\partial u}{\partial \nu}>0$ for any inward directional
		derivative (this means that if $x$ approaches $x_0$ in a ball $B\subset\Omega$
		that has $x_0$ on its boundary, then $\lim\limits_{x\to x_0} \frac{u(x)-u(x_0)}{|x-x_0|}>0$).
	\end{theorem}
	
	In the quasilinear case, the maximum principle is not equivalent to the comparison one. The strong comparison principle for $p$-Laplace equations does not hold in general. However, it is known that the strong comparison principle for $p$-Laplacian holds outside the critical set. This result, which is due to Damascelli \cite{MR1632933}, is enough for our purpose.
	\begin{theorem}[Strong comparison principle]\label{th:scp}
		Let $u,v\in C^1(\Omega)$ be two solutions to
		\[
		-\Delta_p w = f(w) \quad\text{ in } \Omega
		\]
		such that $u \le v$ in $\Omega$, with $p>1$ and let
		\[
		Z=\{x\in\Omega \mid |\nabla u(x)|+|\nabla v(x)|=0\}.
		\]
		If $x_0\in\Omega\setminus Z$ and $u(x_0)=v(x_0)$, then $u=v$ in the connected component of $\Omega\setminus Z$ containing $x_0$.
	\end{theorem}
	
	In companion with the strong comparison principle is the following boundary point lemma \cite[Theorem 2.7.1]{MR2356201}, which holds for $C^2$ solutions.
	\begin{theorem}[Boundary point lemma]\label{lem:bpm}
		Let $u,v\in C^1(\overline{\Omega})\cap C^2(\Omega)$ be solutions of the equation
		\[
		-\Delta_p w = f(w) \quad\text{ in }\Omega
		\]
		such that $u<v$ in $\Omega$ and $u=v$ at some point $x_0 \in \partial \Omega$ where the interior sphere condition is satisfied. Then
		\[
		\frac{\partial u}{\partial \nu}(x_0)<\frac{\partial v}{\partial \nu}(x_0),
		\]
		where $\nu$ is the inward normal at $x_0$.
	\end{theorem}
	
	In the situations where the strong comparison principle does not hold, we will make use of the following weak sweeping principle by Dancer and Yu \cite{MR1955278}:
	\begin{theorem}[Weak sweeping principle]\label{th:wsp}
		Suppose that $\Omega$ is a bounded smooth domain in $\mathbb{R}^N$, $h(x, s)$ is measurable in $x \in \Omega$,
		continuous in $s$, and for each finite interval $J$, there exists a continuous increasing function $L(s)$ such
		that $h(x, s) + L(s)$ is nondecreasing in $s$ for $s \in J$ and $x \in \Omega$. Let $u_t$ and $v_t$, $t \in [t_1, t_2]$, be functions in $W^{1,p}(\Omega) \cap C(\overline \Omega)$ and satisfy in the weak sense,
		\[
		\begin{cases}
			-\Delta_p u_t \ge h(x,u_t) + \varepsilon_1(t) & \text{ in } \Omega,\\
			-\Delta_p v_t \le h(x,v_t) - \varepsilon_2(t) & \text{ in } \Omega,\\
			u_t \ge v_t + \varepsilon & \text{ on } \partial\Omega,\\
		\end{cases}
		\]
		for all $t \in [t_1, t_2]$, where
		\[
		\varepsilon_1(t) + \varepsilon_2(t) \ge \varepsilon > 0.
		\]
		Moreover, suppose that $u_{t_0} \ge v_{t_0}$ in $\Omega$ for some $t_0 \in [t_1, t_2]$ and $t \mapsto u_t$, $t \mapsto v_t$ are continuous from the finite closed interval $[t_1, t_2]$ to $C(\overline \Omega)$. Then
		\[
		u_t \ge v_t \text{ in } \Omega \text{ for all } t \in [t_1, t_2].
		\]
	\end{theorem}
	
	The statement of Theorem \ref{th:wsp} is taken from \cite{MR3058211}. The proof of this theorem is almost identical to that of \cite[Lemma 2.7]{MR1955278}.
	
	\section{Classification of nonnegative solutions in dimension one}\label{sect3}
	We prove Theorem \ref{th:1D} in this section.
	Let $f:[0,+\infty) \to \mathbb{R}$ be a continuous function which is locally Lipschitz continuous except possibly at its zeros such that \eqref{smp_cond} holds. Let $u$ be a solution to \eqref{1D}. In Lemmas \ref{lem:1D1} and \ref{lem:1D2} below, we prove local symmetry properties of $u$.	
	\begin{lemma}\label{lem:1D1}
		If there exist $0\le t_0<t_1<+\infty$ such that $u' \ne 0$ in $(t_0, t_1)$ and $u'(t_1)=0$, then
		\[
		u(t) = u(2t_1-t) \quad\text{ for } t\in(t_0, t_1).
		\]
	\end{lemma}
	\begin{proof}
		We only consider the case $u' > 0$ in $(t_0, t_1)$ since the case $u' < 0$ in $(t_0, t_1)$ is quite similar.
		
		By the standard regularity, $u$ is $C^2$ in $(t_0, t_1)$ and in this interval
		\begin{equation}\label{c2}
			(p-1)(u')^{p-2}u'' = -f(u).
		\end{equation}
		
		First, we claim that $f(u(t_1)) > 0$. Indeed, if $f(u(t_1)) = 0$, then we obtain from \eqref{smp_cond} that
		\[
		-\Delta_p(u(t_1)-u(t)) + C(u(t_1)-u(t))^{p-1} = -f(u(t)) + C(u(t_1)-u(t))^{p-1} \ge 0
		\]
		for some constant $C>0$ and all $t < t_1$ close to $t_1$. We can now apply Hopf's lemma \cite{MR768629} to deduce $u'(t_1)>0$, a contradiction. Now we assume $f(u(t_1)) < 0$, then from \eqref{c2} we obtain $u''(t) > 0$ for all $t < t_1$ and close to $t_1$. It follows that $u'(t_1) > u'(t) > 0$ for $t < t_1$ and close to $t_1$. This contradiction shows that $f(u(t_1)) > 0$.
		
		By continuity, we have $f(u(t)) > 0$ for $t\in[t_1,t_1+\varepsilon]$ for some small $\varepsilon>0$.
		
		We claim that $u'(t)\le0$ or all $t\in(t_1,t_1+\varepsilon)$. Otherwise we can find $r_1\in(t_1,t_1+\varepsilon)$ such that $u'(r_1) > 0$. Consider the maximal interval $(r_2,r_1]\subset(t_1,r_1]$ such that
		\[
		u'(t)>0 \text{ in } (r_2,r_1], \quad u'(r_2)=0.
		\]
		From the standard elliptic regularity, we know that $u$ is $C^2$ in $(r_2,r_1]$. Hence \eqref{c2} holds in this interval and $u''(t)<0$ for $t\in(r_2,r_1]$. This implies $u'(r_2)>u'(r_1)>0$, which is a contradiction.
		
		We further claim that $u'(t)<0$ or all $t\in(t_1,t_1+\varepsilon)$. Suppose on the contrary that $u'(r_1) = 0$ for some $r_1\in(t_1,t_1+\varepsilon)$. If $u'(t) = 0$ for all $t\in[t_1,r_1]$ then from the equation we deduce $f(u(t)) = 0$ in this interval, a contradiction. Hence there exists $r_2\in(t_1,r_1)$ such that $u'(r_2) < 0$. Let $(r_2,r_3)\subset(r_2,r_1)$ be the maximal interval such that
		\[
		u'(t)<0 \text{ in } [r_2,r_3), \quad u'(r_3)=0.
		\]
		Now, \eqref{c2} holds in $[r_2,r_3)$, which implies $u''(t)<0$ in this interval. Therefore, $u'(r_2)>u'(r_3)=0$, a contradiction.
		
		It follows from \eqref{c2} that
		\begin{equation}\label{p-Laplace-integral}
			\frac{p-1}{p} |u'(t)|^p + F(u(t)) = F(u(t_1)) \quad\text{ for } t\in[t_0,t_1].
		\end{equation}		
		Therefore,
		\begin{equation}\label{unique}
			\int_{u(t_0)}^{u(t)} \frac{ds}{\left[F(u(t_1))-F(s)\right]^\frac{1}{p}} = \left(\frac{p}{p-1}\right)^\frac{1}{p}(t-t_0) \quad\text{ for all } t\in[t_0,t_1].
		\end{equation}
		
		Now let $(t_1,t_2)\subset(t_1,+\infty)$ be the maximal interval such that $u'(t)<0$ or all $t\in(t_1,t_2)$. Hence $t_2\ge t_1+\varepsilon$ and either $t_2=+\infty$ or $u'(t_2)=0$. We show that $t_2\ge2t_1-t_0$ and $u(t)=u(2t_1-t)$ for $t\in [t_0,t_1]$. Clearly, \eqref{p-Laplace-integral} also holds for $t\in[t_1, t_2]$ and we deduce
		\begin{equation}\label{unique2}
			\int_{u(t)}^{u(t_1)} \frac{ds}{\left[F(u(t_1))-F(s)\right]^\frac{1}{p}} = \left(\frac{p}{p-1}\right)^\frac{1}{p}(t-t_1) \quad\text{ for all } t\in[t_1,t_2].
		\end{equation}
		From \eqref{unique} and \eqref{unique2}, we have
		\begin{align*}
			\int_{u(t_0)}^{u(t)} \frac{ds}{\left[F(u(t_1))-F(s)\right]^\frac{1}{p}} &= \left\{\int_{u(t_0)}^{u(t_1)} - \int_{u(t)}^{u(t_1)}\right\} \frac{ds}{\left[F(u(t_1))-F(s)\right]^\frac{1}{p}}\\
			&= \left(\frac{p}{p-1}\right)^\frac{1}{p}(2t_1-t_0-t) \quad\text{ for all } t\in[t_1,t_2],
		\end{align*}
		or equivalently,
		\[
		\int_{u(t_0)}^{u(2t_1-t)} \frac{ds}{\left[F(u(t_1))-F(s)\right]^\frac{1}{p}} = \left(\frac{p}{p-1}\right)^\frac{1}{p}(t-t_0) \quad\text{ for all } t\in[2t_1-t_2, t_1],
		\]
		Comparing this with \eqref{unique} we immediately obtain $t_2\ge2t_1-t_0$ and $u(2t_1 - t) = u(t)$ in $[t_0, t_1]$.
	\end{proof}
	
	\begin{lemma}\label{lem:1D2}
		If there exist $0< t_0<t_1\le+\infty$ such that $u' \ne 0$ in $(t_0, t_1)$ and $u'(t_0)=0$, then $t_0\ge\frac{t_1}{2}$ (hence $t_1<+\infty$) and
		\[
		u(t) = u(2t_0-t) \quad\text{ for } t\in(t_0, t_1).
		\]
	\end{lemma}
	
	\begin{proof}
		We only consider the case $u' > 0$ in $(t_0, t_1)$ since the other case is similar.
		
		Arguing as in the proof of Lemma \ref{lem:1D1}, we can show that $f(u(t_0))<0$ and $u'(t)<0$ for all $t\in(t_0-\varepsilon,t_0)$ for some $\varepsilon>0$ small. Then taking into account
		\[
		u(t) > 0 = u(0) \quad\text{ for } t\in(t_0,t_1)
		\]
		and arguing as before, we obtain that $t_1<+\infty$, $t_0\ge t_1-t_0$ and $u$ is symmetric in $(2t_0-t_1, t_1)$ about $t=t_0$.
	\end{proof}
	
	With Lemmas \ref{lem:1D1} and \ref{lem:1D2} in force, we are in a position to prove Theorem \ref{th:1D}.
	\begin{proof}[Proof of Theorem \ref{th:1D}]
		Let $u\not\equiv 0$ be a solution to \eqref{1D}.
		
		Since $u(0)=0$ and $u$ is positive somewhere, we have $u'(\tilde t)>0$ for some $\tilde t>0$. Let $(t_0, t_1)\subset(0,+\infty)$ be the largest interval containing $\tilde t$ such that $u'(t) > 0$ for all $t\in(t_0, t_1)$. We consider two cases.
		
		\textit{Case 1:} $t_1=+\infty$. From Lemma \ref{lem:1D2}, we deduce
		\[
		t_0=0.
		\]
		
		Hence $u'(t) > 0$ for all $t\in(0,+\infty)$. Consequently, $u \in C^2((0,+\infty))$ and
		\[
		(p-1)(u')^{p-2}u'' = -f(u) \text{ in } (0,+\infty).
		\]
		This implies
		\begin{equation}\label{equiv_eq}
			\frac{p-1}{p} (u')^p + F(u) = C \quad\text{ for } t\in[0,+\infty),
		\end{equation}
		where $C$ is a constant. Moreover, since $u$ is bounded and increasing, the limit
		\[
		\rho = \lim_{t\to+\infty} u(t)
		\]
		exists and is finite. For the same reason, we can choose a sequence $t_n\to+\infty$ such that $u'(t_n)\to0$. Integrating $-((u')^{p-1})' = f(u)$ in $(0,t_n)$, we get
		\begin{equation}\label{ks}
			(u')^{p-1}(0) - (u')^{p-1}(t_n) = \int_{0}^{t_n} f(u(s)) ds.
		\end{equation}
		It follows that $f(\rho)=0$. Otherwise, the right-hand side of \eqref{ks} goes to $\pm\infty$, while the left-hand side is finite as $n\to\infty$.
		Letting $t=t_n\to+\infty$ in \eqref{equiv_eq}, we deduce $C=F(\rho)$. Hence
		\begin{equation}\label{kl}
			\frac{p-1}{p} (u')^p + F(u) = F(\rho) \quad\text{ for } t\in[0,+\infty).
		\end{equation}
		In particular, $u'(0)=\left(\frac{p}{p-1}F(\rho)\right)^\frac{1}{p}$ and $F(\rho)>F(u(t))$ for $t\in(0,+\infty)$. Since $u:(0,+\infty)\to(0,\rho)$ is a bijective function, we obtain
		\[
		F(\rho)>F(t) \quad\text{ for } t\in(0,\rho).
		\]
		
		Moreover, if $f(0)\ge0$, then from \eqref{smp_cond} and Hopf's lemma (Theorem \ref{th:smp}), we have $u'(0)>0$. Hence using \eqref{kl} with $t=0$, we deduce $F(\rho)>F(0)$.
		
		Conversely, let $\rho\in Z_f^*$. From \eqref{smp_cond}, we deduce
		\[
		F(\rho) - F(s) \le C(\rho - s)^p
		\]
		for some $C>0$ and all $s<\rho$ close to $\rho$. This implies that
		\begin{equation}\label{k2}
			\int_{0}^{\rho} \frac{ds}{\left[F(\rho)-F(s)\right]^\frac{1}{p}} = +\infty.
		\end{equation}
		On the other hand,
		\begin{equation}\label{k3}
			\int_{0}^{t} \frac{ds}{\left[F(\rho)-F(s)\right]^\frac{1}{p}} < +\infty \quad\text{ for } 0<t<\rho.
		\end{equation}
		Indeed, if $F(\rho)>F(t)$ for $t\in[0,\rho)$, then \eqref{k3} holds. If $F(\rho)=F(0)>F(t)$ for $t\in(0,\rho)$, which only occurs in the case $f(0)<0$, then we have
		\[
		F(\rho)-F(s) = -F(s) = \int_{0}^{s} (-f(s))ds > \frac{|f(0)|s}{2}
		\]
		for small $s$ and \eqref{k3} also holds.
		
		Hence problem \eqref{1D} admits a solution $u_\rho$ which is uniquely determined by the formula
		\[
		\int_{0}^{u_\rho(t)} \frac{ds}{\left[F(\rho)-F(s)\right]^\frac{1}{p}} = \left(\frac{p}{p-1}\right)^\frac{1}{p}t \quad\text{ for all } t\in[0,+\infty).
		\]
		Taking the derivative, we can easily check that $u_\rho'>0$ and $u_\rho'(0)=\left(\frac{p}{p-1}F(\rho)\right)^\frac{1}{p}$. Moreover, \eqref{k2} implies $\lim_{t\to+\infty} u_\rho=\rho$.
		
		\textit{Case 2:} $t_1<+\infty$. In this case $u'(t_1)=0$.
		
		We show that $u'(t_0)=0$. This is obvious if $t_0>0$. Suppose on the contrary that $t_0=0$ and $u'(0)>0$. Then we can apply Lemma \ref{lem:1D1} to obtain
		\[
		u(t) = u(2t_1-t) \quad\text{ for } t\in(0, t_1).
		\]
		Consequently, $u(2t_1)=u(0)=0$ and $u'(2t_1)=-u'(0)<0$. This implies $u(t)<0$ for $t>2t_1$ and close to $2t_1$, a contradiction.
		
		Hence $u'(t_0)=0$ and $u$ is periodic due to Lemmas \ref{lem:1D1} and \ref{lem:1D2}.		
		More precisely, by setting $t^*=t_1-t_0$, we have
		$u'>0$ in $(0,t^*)$, $u'(0)=u'(t^*)=0$ and
		\[
		u(t+2kt^*)=u(t), \quad u(t+(2k+1)t^*)=u(t^*-t)
		\]
		for all $t\in[0,t^*]$ and all $k\in\mathbb{N}$. We must have $f(0)<0$. Otherwise, Hopf's lemma (Theorem \ref{th:smp}) would imply $u'(0)>0$, a contradiction. From the proof of Lemma \ref{lem:1D1}, we also have $f(u(t^*))>0$.
		
		Setting $\rho=u(t^*)$, from \eqref{c2} we have
		\begin{equation}\label{k4}
			\frac{p-1}{p} (u')^p + F(u) = F(\rho) \quad\text{ for } t\in[0,t^*],
		\end{equation}
		
		Hence $F(\rho)=F(0)>F(u(t))$ for $t\in(0,t^*)$. Since $u:(0,t^*)\to(0,\rho)$ is a bijective function, this implies $F(\rho)=F(0)>F(t)$ for $t\in(0,\rho)$. Moreover, \eqref{k4} yields
		\begin{equation}\label{k5}
			\int_{0}^{u(t)} \frac{ds}{\left[F(\rho)-F(s)\right]^\frac{1}{p}} = \left(\frac{p}{p-1}\right)^\frac{1}{p}t \quad\text{ for all } t\in[0,t^*].
		\end{equation}
		
		Let $t=t^*$, we find the relation between $t^*$ and $\rho$
		\[
		t^* = \left(\frac{p-1}{p}\right)^\frac{1}{p}\int_{0}^{\rho} \frac{ds}{\left[F(\rho)-F(s)\right]^\frac{1}{p}}.
		\]
		Hence $u$ has the form $\tilde u_\rho$. Conversely, if $f(0)<0$ and $\rho\in P_f$, then
		\[
		F(\rho)-F(s) = -F(s) = \int_{0}^{s} (-f(s))ds > \frac{|f(0)|s}{2}
		\]
		for small $s$. Hence
		\[
		\int_{0}^{\rho} \frac{ds}{\left[F(\rho)-F(s)\right]^\frac{1}{p}} < +\infty.
		\]
		Therefore, problem \eqref{1D} admits a $t^*$-periodic solution $u$ which is uniquely determined by formula \eqref{k5}.
		
		This completes the proof.
	\end{proof}
	
	\section{Existence of positive solutions in a ball}\label{sect4}
	In this section, we analyze the $N$-dimensional problem
	\begin{equation}\label{ball}
		\begin{cases}
			-\Delta_p u = f(u) &\text{ in } B_r,\\
			u > 0 &\text{ in } B_r,\\
			u = 0 &\text{ on } \partial B_r,
		\end{cases}
	\end{equation}
	where $B_r \subset \mathbb{R}^N $ stands for the ball of radius $r$ centered at the origin. Motivated by \cite{MR937538,MR2822222,MR3058211}, we have the following existence result.
	\begin{proposition}\label{prop:ball}
		Assume that $f$ is a locally Lipschitz continuous function with $f (0) \ge 0$ and $f(\rho) = 0$ for some $\rho > 0$. Assume in addition that
		\[
		F_\rho(t):=\int_{t}^{\rho} f(s) ds > 0 \quad\text{ for all } t \in [0, \rho).
		\]
		Then for every $\varepsilon > 0$ there exists a positive number
		$R_0 = R_0(\varepsilon)$ such that for $r \ge R_0$, problem \eqref{ball} admits a solution $u_r \in C^{1,\alpha}(\overline{B_r})$ such that
		\[
		\rho - \varepsilon \le \|u_r\|_{L^\infty(B_r)} < \rho.
		\]
	\end{proposition}
	\begin{proof}
		We first observe that for each $n>1$, there exists $\varepsilon<\frac{1}{n}$ such that
		\begin{equation}\label{F_rho}
			F_\rho(t) \ge F_\rho(\rho-\varepsilon) \quad\text{ for all } t\in[0,\rho-\varepsilon].
		\end{equation}
		Indeed, we may choose
		\[
		\varepsilon = 
		\begin{cases}
			\frac{1}{n} & \text{ if } \min_{[0,\rho-\frac{1}{n}]} F_\rho = F_\rho(\rho-\frac{1}{n}),\\
			\sup\{z<\frac{1}{n} \mid F_\rho(\rho-z)\le \min_{[0,\rho-\frac{1}{n}]} F_\rho \} & \text{ otherwise}.
		\end{cases}
		\]
		Therefore, it suffices to prove the proposition for $\varepsilon>0$ sufficiently small so that \eqref{F_rho} is satisfied. We set
		\[
		\tilde{f}(t) = 
		\begin{cases}
			f(0) &\text{ for } t<0,\\
			f(t) &\text{ for } 0\le t\le\rho,\\
			0 &\text{ for } t>\rho,
		\end{cases}
		\]
		and
		\[
		\tilde{F}_\rho(t) = \int_{t}^{\rho} \tilde{f}(s) ds,
		\]
		then $\tilde{F}_\rho(t)\ge0$ for all $t\in\mathbb{R}$. We consider the energy functional
		\[
		I_r(u) = \int_{B_r} \left(\frac{1}{p}|\nabla u|^p + \tilde{F}_\rho(u)\right)
		\]
		for $u\in W^{1,p}_0(B_r)$. Clearly, a critical point of $I_r$ corresponds to a weak solution of
		\begin{equation}\label{ball_cut}
			\begin{cases}
				-\Delta_p u = \tilde{f}(u) &\text{ in } B_r,\\
				u = 0 &\text{ on } \partial B_r.
			\end{cases}
		\end{equation}
		Because $\tilde{f}(t)\ge0$ for $t\le0$ and $\tilde{f}(t)=0$ for $t\ge\rho$, we may exploit the weak maximum principle to deduce that any such a solution satisfies
		\begin{equation}\label{range}
			0 \le u \le \rho.
		\end{equation}
		More precisely, we may directly test \eqref{ball_cut} with $\max\{-u,0\}$ and $\max\{u-\rho,0\}$ to get \eqref{range}.
		Therefore, all solutions of \eqref{ball_cut} are also solutions of \eqref{ball}. Moreover, by standard elliptic regularity for $p$-Laplace equations (see \cite{MR709038,MR727034,MR969499}), we know that such solutions belong to $C^{1,\alpha}(\overline{B_r})$.
		
		Since the functional $I_r$ is well-defined and coercive, it has a minimizer $u_r$. Thus, $u_r$ is a nonnegative solution to \eqref{ball}. Moreover, since $u_r$ is a minimizer, by the well-known rearrangement theory (see \cite{MR810619}), it must be radially symmetric and decreasing away from the center of the domain. By abuse of notation, we write $u_r(x)=u_r(|x|)$. Then
		\[
		0 \le u_r(|x|) \le u_r(0) \le \rho \quad\text{ in } B_r.
		\]
		
		It remains to show that there exists $r>0$ such that $u_r(0) \ge \rho - \varepsilon$. Otherwise $u_r \le \rho - \varepsilon$ for all $r>0$. Exploiting \eqref{F_rho}, we obtain
		\[
		I_r(u_r) \ge \int_{B_r} F_\rho(u_r) \ge \int_{B_r} F_\rho(\rho-\varepsilon) = \omega_N r^N F_\rho(\rho-\varepsilon) \quad\text{ for all } r>0,
		\]
		where $\omega_N$ denotes the volume of $B_1$. We consider
		\[
		w_r(x) = 
		\begin{cases}
			\rho &\text{ for } |x|<r-1,\\
			\rho(r-|x|) &\text{ for } r-1<|x|\le r.
		\end{cases}
		\]
		Clearly, $w_r \in W^{1,p}_0(B_r)$ and $|\nabla w_r|$ and $F_\rho(w_r)$ are supported on the annulus $\{r-1 \le |x| \le r\}$. Hence there exists a constant $C$ independent of $r$ such that
		\[
		I_r(w_r) \le C\left[r^N - (r-1)^N\right] \quad\text{ for all } r>1.
		\]
		Since $v_r$ is the minimizer of $I_r$, we have $I_r(u_r) \le I_r(w_r)$. Thus
		\[
		\omega_N r^N F_\rho(\rho-\varepsilon) \le C\left[r^N - (r-1)^N\right] \quad\text{ for all } r>1,
		\]
		or equivalently,
		\[
		\omega_N F_\rho(\rho-\varepsilon) \le C\left[1 - \left(\frac{r-1}{r}\right)^N\right] \quad\text{ for all } r>1.
		\]
		Since $F_\rho(\rho-\varepsilon)>0$, the above inequality does not hold for large $r$. This contradiction shows that $u_r(0) \ge \rho-\varepsilon$ for all large $r$.
	\end{proof}
	
	\section{Monotonicity and one-dimensional symmetry}\label{sect5}

	\begin{lemma}\label{lem:upper}
		Under assumptions of Theorem \ref{th:main}, there exists a one-dimensional solution $\overline u$ of \eqref{main} such that $u \le \overline u \le \rho$.
	\end{lemma}
	
	\begin{proof}
		We show that there is a maximal solution $\overline u$ in the order interval $[u, \rho]$ in the sense that any
		solution $v$ with $u\le v\le \rho$ satisfies $v \le \overline u$. To this end, for each $n>1$, we define
		\[
		B_n^+ = \mathbb{R}^N_+ \cap B_n.
		\]
		
		Setting
		\[
		\phi_n(x) =
		\begin{cases}
			\rho(|x'|-n+1)_+ &\text{ if } x_N=0,\\
			\rho &\text{ if } x_N>0,
		\end{cases}
		\]
		where we denote $t_+=\max\{t,0\}$.
		We consider the auxiliary problem
		\begin{equation}\label{aux1}
			\begin{cases}
				-\Delta_p u = f(u) &\text{ in } B_n^+,\\
				u = \phi_n &\text{ on } \partial B_n^+.
			\end{cases}
		\end{equation}
		
		Obviously, $u|_{B_n^+}$ is a lower solution and $\rho$ is an upper solution to \eqref{aux1}. Hence by the standard upper and lower solution argument, problem \eqref{aux1} has a maximal solution $u_n$ in the order interval $[u,\rho]$. By standard regularity, $u_n\in C^{1,\alpha}(B_n^+)$.
		By the maximality, we have $u_n \ge u_{n+1}$ in $B_n^+$ and $\overline u := \lim_{n\to\infty} u_n$ is a solution of \eqref{main}.
		
		Clearly, $u \le \overline u \le \rho$.
		If $v$ is any solution of \eqref{main} in $[u, \rho]$, then $v|_{B_n^+}$ is a lower solution of \eqref{aux1} and it follows that $u_n \ge v$ (since the standard upper and lower solution argument implies the existence of a solution in $[v,\rho]$ which is less than or equal to the maximal solution $u_n$). It follows that $v \le \overline u$. Thus $\overline u$ is the maximal
		solution in $[u, \rho]$.
		
		Since the equation is invariant under translations in $x'$, the maximality implies that $\overline u$ is a function of $x_N$ only.
		Indeed, we assume by contradiction that there exist two points $x=(x^0,l)$ and $y=(y^0,l)$ with $l>0$, such that $\overline u(x)\ne \overline u(y)$. We define $\tilde{u}(x',y)=\overline u(x'+y^0-x^0, x_N)$. Then $\tilde{u}$ is a solution to \eqref{main} and $v(x):=\max\{\overline u(x),\tilde{u}(x)\} \ge \overline u(x)$. Moreover, $v\not\equiv \overline u$ and $v$ is a lower solution to \eqref{main} satisfying $v\le \rho$. Hence \eqref{main} has a solution $u^*$ in the order interval $[v,\rho]$. It follows that $u^*$ is a solution in $[u,\rho]$ which satisfies $u^*\ge \overline u$ and $u^*\not\equiv \overline u$. This contradicts the maximality of $\overline u$.
		
		Thus $\overline u$ is a function of $x_N$ only.
	\end{proof}
	
	\begin{lemma}\label{lem:lower}
		Under assumptions of Theorem \ref{th:main}, there exists a one-dimensional solution $\underline u$ of \eqref{main} such that $\|\underline u\|_{L^\infty(\mathbb{R}^N)}=\rho$ and $\underline u \le u$ in $\mathbb{R}^N$.
	\end{lemma}
	\begin{proof}
		For $\alpha<\beta$, we denote
		\[
		\Sigma_{\alpha,\beta} = \{x\in\mathbb{R}^N \mid \alpha < x_N < \beta\} \quad\text{ and }\quad \Sigma_\alpha = \Sigma_{\alpha,\infty} := \{x\in\mathbb{R}^N \mid \alpha < x_N \}.
		\]
		
		We will construct lower solution $\underline u$ in 3 steps:
		
		\textit{Step 1.} For every $r > 0$ and every $\varepsilon > 0$ small enough, there exists $x_0 \in \mathbb{R}^N_+$
		such that $B_r(x_0) \subset \overline\Sigma_1$ and $u \ge \rho - \varepsilon$ in $B_r(x_0)$.
		
		To prove this assertion we take a sequence of points $x_n=(x_n', x_{n,N}) \in \mathbb{R}^N_+$ such that $\lim_{n\to\infty} u(x_n) = \rho$. Since the right hand side of \eqref{main} is bounded, by standard elliptic estimates, we have $u\in C^{1,\alpha}(\overline{\mathbb{R}^N_+})$. Because of the Dirichlet condition, this implies that $(x_{n,N})$ is bounded away from zero. Hence extracting a subsequence we may assume that either $x_{n,N}\to y_0$ for some $y_0>0$ or $x_{n,N}\to+\infty$. We define
		\[
		u_n(x) = u(x+x_n).
		\]
		
		Passing to a subsequence, we have $u_n\to v$ in $C^{1,\alpha'}(\overline{\mathbb{R}^N_+})$, where $v$ is a solution of
		\[
		\begin{cases}
			-\Delta_p v = f(v) & \text{ in } \Sigma,\\
			0 \le v \le \rho & \text{ in } \Sigma,
		\end{cases}
		\]
		where $\Sigma=\{x_N \ge -y_0\}$ in case $x_{n,N}\to y_0$ or $\Sigma=\mathbb{R}^N$ when $x_{n,N}\to+\infty$. In either case, using \eqref{smp_cond} and the fact that $f (\rho) = 0$ and $v(0)=\rho$, the strong maximum principle (Theorem \ref{th:smp}) implies $v\equiv\rho$. However, if $x_{n,N}\to y_0$, then from the Dirichlet condition on $u$, we have $v=0$ on $\partial \Sigma$, a contradiction.
		
		Therefore, the case $x_{n,N}\to+\infty$ must happen and $u(x+x_n) \to \rho$ locally uniformly in $\mathbb{R}^N$.
		
		Consequently, for any $\varepsilon > 0$ and $r > 0$, we have $u(x + x_n) \ge \rho - \varepsilon$ in $B_r$ if $n$ is larger than some $n_0 = n_0(\varepsilon, r)$. Then $u(x) \ge \rho - \varepsilon$ in $B_r(x_n)$ for those values of $n$.
		
		\textit{Step 2.} For every $\varepsilon > 0$, there exists $\eta > 0$ such that $u(x) \ge \rho - \varepsilon$ for $x\in\overline\Sigma_\eta$.
		
		By Lemma \ref{lem:upper}, problem \eqref{main} has a one-dimensional solution $\overline u$ such that $\sup_{\mathbb{R}^N_+} \overline u = \rho$ and $u \le \overline u$. Since $\frac{\partial\overline u}{\partial x_N}>0$, the strong comparison principle (Theorem \ref{th:scp}) can be applied to yield that either $u \equiv \overline u$ or $u < \overline u$. If the former case happens, then the lemma is proved by setting $\underline u = \overline u$.
		
		Hence in what follows, we assume $u < \overline u$ in $\mathbb{R}^N_+$. We claim that
		\begin{equation}\label{s1}
			F(t) < F(\rho) \quad\text{ for all } t\in[0,\rho).
		\end{equation}	
		Indeed, if (i) or (ii) holds, then \eqref{s1} follows from Theorem \ref{th:1D}.
		If (iii) holds, then both $u$ and $\overline u$ belong to $C^2(\mathbb{R}^N_+\cap B_r(x_0))$. Hence we can apply the boundary point lemma (Theorem \ref{lem:bpm}) to obtain
		\[
		\frac{\partial \overline u}{\partial x_N}>\frac{\partial u}{\partial x_N}\ge0.
		\]
		Therefore, an application of Theorem \ref{th:1D} also yields \eqref{s1}.
		
		If $f (0) < 0$, we choose $\mu>0$ small such that
		\begin{equation}\label{s2}
			\frac{f(0)}{2}\mu + F(\rho) > 0.
		\end{equation}
		Then we define function $f_\mu:[-\mu,\rho]\to\mathbb{R}$ by setting
		\[
		f_\mu(t) =
		\begin{cases}
			\frac{f(0)}{\mu}(t+\mu) & \text{ if } t\in[-\mu,0),\\
			f(t) & \text{ if } t\in[0,\rho].
		\end{cases}
		\]
		In the case $f (0) \ge 0$, we simply take $\mu = 0$ and $f_0 = f$. From \eqref{s1} and \eqref{s2}, we deduce
		\begin{equation}\label{s3}
			\int_{t}^{\rho} f_\mu(s) ds > 0 \quad\text{ for all } t\in[-\mu,\rho).
		\end{equation}
		
		Now for each small $\delta>0$, we define
		\[
		g_{\mu,\delta}(t) := f_\mu(t) - \delta (t+\mu) \quad\text{ in } [-\mu,\rho].
		\]
		
		Since $f$ has no zeros in a left neighborhood of $\rho$ and \eqref{s1} holds, we deduce that $f$ is positive in that neighborhood. On the other hand, $\int_{t}^{\rho} g_{\mu,\delta}(s) ds \to \int_{t}^{\rho} f_\mu(s) ds$ as $\delta\to0$, uniformly in $t\in[-\mu,\rho]$. Hence taking into account \eqref{s3}, for sufficiently small $\delta$, there exists $\rho_\delta\nearrow\rho$ (as $\delta\to0$) such that $g_{\mu,\delta}(\rho_\delta)=0$ and
		\[
		G(t):=\int_{t}^{\rho_\delta} g_{\mu,\delta}(s) ds > 0 \quad\text{ for all } t \in [-\mu, \rho_\delta).
		\]
		
		We can therefore apply Proposition \ref{prop:ball} to deduce that the problem
		\[
		\begin{cases}
			-\Delta_p w = g_{\mu,\delta}(w-\mu) &\text{ in } B_r,\\
			w = 0 &\text{ on } \partial B_r
		\end{cases}
		\]
		with some $r=r(\delta)$ has a positive solution $w_r$ verifying
		\[
		\rho_\delta + \mu - \delta \le \|w_r\|_{L^\infty(B_r)} < \rho_\delta + \mu.
		\]
		
		Hence by setting $v_r=w_r-\rho$, we have that $v_r\ge-\mu$,
		\[
		\rho_\delta - \delta \le \|v_r\|_{L^\infty(B_r)} < \rho_\delta
		\]
		and $v_r$ solves the problem
		\[
		\begin{cases}
			-\Delta_p v = g_{\mu,\delta}(v) &\text{ in } B_r,\\
			v = -\mu &\text{ on } \partial B_r.
		\end{cases}
		\]
		
		By Step 1, we can find $x_0$ so that
		$B_r(x_0) \subset \overline\Sigma_1$ and $u \ge \rho_\delta$ in $B_r(x_0)$.
		
		Let $\tilde{x}$ be an arbitrary point in $\overline\Sigma_{r+1}$. We define $x^t = (1-t)x^0 + t\tilde{x}$ and $u^t(x)=u(x+x^t)$.
		Clearly, $B_r(x^t) \subset \overline\Sigma_1$ for all $t \in [0, 1]$.
		
		Since $u > 0$ on the compact set $\bigcup_{t\in[0,1]} \overline{B_r(x^t)}$, we can find $\lambda > 0$ such that $u \ge \lambda$ on this set. Let $r_1 < r$ be such that $-\mu < v_r < \frac{\lambda}{2}$ on $\{|x| = r_1\}$. Then for all $t \in [0, 1]$, we have
		\[
		\begin{cases}
			-\Delta_p u^t = f(u^t) & \text{ in } B_{r_1},\\
			-\Delta_p v_r = g_{\mu,\delta}(v_r) \le f(v_r) - \gamma & \text{ in } B_{r_1},\\
			u^t \ge v_r + \frac{\lambda}{2} & \text{ on } \partial B_{r_1},
		\end{cases}
		\]
		where $\gamma=\inf_{x\in B_{r_1}} [f (v_r(x)) - g_{\mu,\delta}(v_r(x))] > 0$. Moreover, $u^0 = u \ge \rho_\delta \ge v_r$ in $B_{r_1}$. Thus we can apply the weak sweeping principle (Theorem \ref{th:wsp}) to deduce that $u^t \ge v_r$ in $B_{r_1}$ for all $t \in [0, 1]$. In particular, $u(\tilde{x}) \ge v_r(0) \ge \rho_\delta - \delta$. Since $\tilde{x} \in \overline\Sigma_{r+1}$ is arbitrary, we have
		\[
		u \ge \rho_\delta - \delta \quad \text{ in } \overline\Sigma_{r+1}.
		\]
		Since $\rho_\delta\to\rho$ as $\delta\to0$, we can choose $\delta$ small such that $\rho_\delta - \delta > \rho - \varepsilon$ for given $\varepsilon>0$. Hence the conclusion follows.
		
		\textit{Step 3.} There exists a minimal solution $\underline u$ of \eqref{main} such that $\|\underline u\|_{L^\infty(\mathbb{R}^N)}=\rho$ and $\underline u \le u$ in $\mathbb{R}^N$.
		
		Since $f$ has no zeros in a left neighborhood of $\rho$ and \eqref{s1} holds, there exists a small $\delta > 0$ such that $f > 0$ in $[\rho - \delta, \rho)$. Now we choose $\varepsilon<\delta$ sufficiently small and then modify $f$ over $[0, \rho]$ to obtain a new function $f_\varepsilon$ such that
		\[
		f_\varepsilon = f \text{ in } [0, \rho-\delta],\quad
		0<f_\varepsilon<f \text{ in } (\rho-\delta,\rho-\varepsilon),\quad
		f_\varepsilon(\rho-\varepsilon)=0
		\]
		and
		\[
		\int_{t}^{\rho-\varepsilon} f_\varepsilon(s) ds > 0 \quad\text{ for all } t\in[0, \rho-\varepsilon).
		\]
		
		Hence we can apply Theorem \ref{th:1D} to see that the problem
		\[
		\begin{cases}
			-\Delta_p v = f_\varepsilon(v) & \text{ in } \mathbb{R}_+,\\
			v(0)=0,~ \lim_{t\to+\infty} v(t) = \rho-\varepsilon
		\end{cases}
		\]
		has a unique positive solution $v_\varepsilon$, which is monotone increasing. For any $\eta>0$, we define
		\[
		v_{\varepsilon,\eta}(t) =
		\begin{cases}
			0 & t \in [0,\eta),\\
			v_\varepsilon(t-\eta) & t \in [\eta,+\infty).
		\end{cases}
		\]
		
		It is easy to check that $v_{\varepsilon,\eta}$ is a lower solution of \eqref{main}. By Step 2, we can find a large number $\eta$ so that $u>\rho-\varepsilon$ in $\{x_N\ge \eta\}$. Hence in view of $v_\varepsilon\le \rho-\varepsilon$, we find that $v_{\varepsilon,\eta}\le u$ in $\mathbb{R}^N_+$ for such $\eta$. For $n$ large, using the auxiliary problem
		\[
		\begin{cases}
			-\Delta_p u = f(u) &\text{ in } B_n^+,\\
			u = (\rho-\varepsilon)\chi_{\{x_N\ge\eta\}} &\text{ on } \partial B_n^+
		\end{cases}
		\]
		and arguing as in the proof of Lemma \ref{lem:upper}, we find that equation \eqref{main} has a minimal solution $\underline u$ in the order interval $[v_{\varepsilon,\eta}, u]$.
		
		Moreover, the minimality of $\underline u$ implies that $\underline u$ is a function of $x_N$ only. Thus $\underline u$ is a solution to \eqref{1D}. Due to the lower bound $v_{\varepsilon,\eta}$, we see that $\underline u$ is not periodic. Hence it must be positive and monotone by Theorem \ref{th:1D}. Moreover, $\lim_{t\to+\infty} \underline u(t)$ is a zero of $f$ contained in $[\rho-\varepsilon,\rho]$. But $\rho$ is the only zero of $f$ in this range. Hence $\lim_{t\to+\infty} \underline u(t)=\rho$. This implies $\|\underline u\|_{L^\infty(\mathbb{R}^N)}=\rho$.
	\end{proof}
	
	We are ready to conclude Theorems \ref{th:main} and \ref{th:sub}.
	\begin{proof}[Proof of Theorem \ref{th:main}]
		By Lemmas \ref{lem:upper} and \ref{lem:lower}, there exist two one-dimensional positive solutions $\overline u$ and $\underline u$ such that $\underline u \le u \le \overline u$ in $\mathbb{R}^N$. Moreover, $\|\underline u\|_{L^\infty(\mathbb{R}^N)}=\|\overline u\|_{L^\infty(\mathbb{R}^N)}=\rho$. Therefore, we can apply Theorem \ref{th:1D} to deduce $\underline u(x) = \overline u(x) = u_\rho(x_N)$.
		
		Hence we have $u = u_\rho(x_N)$. Furthermore, \eqref{deri} follows from properties of $u_\rho$ stated in Theorem \ref{th:1D}. This completes the proof of Theorem \ref{main}.
	\end{proof}
	
	\begin{proof}[Proof of Theorem \ref{th:sub}]
		According to \cite[Theorem 1.7]{MR2886112}, it can be concluded that $0 < u \le \rho$.		
		Due to the behavior of the nonlinearity near \( \rho \) (see (ii)), the strong maximum principle applies, which leads to the conclusion that \( 0 < u < \rho \) in $\mathbb{R}^N_+$. Following the same reasoning as in the proof of \cite[Theorem 1.3]{MR3303939}, it can be inferred that \( u \) is strictly bounded away from \( \rho \) in \( \Sigma_\lambda \) for any \( \lambda > 0 \). Moreover, the standard gradient estimate gives $|\nabla u|\in L^\infty(\mathbb{R}^N_+)$. Now we can apply the results in \cite{MR3303939,MR3752525} we deduce that
		\[
		\frac{\partial u}{\partial x_N} \ge 0 \text{ in } \mathbb{R}^N_+.
		\]
		In fact, it's worth noting that $f>0$ within the range of values the solution takes in any strip, which is sufficient to reapply the moving plane method as in \cite{MR3303939,MR3752525}.
		
		Now the function
		\[
		v(x'):=\lim_{x_N\to+\infty} u(x',x_N)
		\]
		is well-defined and satisfies
		\[
		-\Delta_p v = f(v) \quad\text{ in } \mathbb{R}^{N-1}.
		\]
		
		We claim that $\sup_{\mathbb{R}^{N-1}} v = \rho$. We argue by contradiction. Suppose $\sup_{\mathbb{R}^{N-1}} v = t_1 < \rho$. Then $v\le t_1$ in $\mathbb{R}^{N-1}$. Assumptions (i) and (iii) implies
		\[
		f(t) > c_1 t^\gamma \text{ in } (0, t_1)
		\]
		for some $c_1>0$.
		Hence $v$ satisfies
		\[
		\begin{cases}
			-\Delta_p v \ge c_1 v^\gamma &\text{ in } \mathbb{R}^{N-1},\\
			v > 0 &\text{ in } \mathbb{R}^{N-1},
		\end{cases}
		\]
		This contradicts a Liouville-type theorem by E. Mitidieri and S.I. Pokhozhaev \cite{MR1784317}. Therefore, \(\sup_{\mathbb{R}^{N-1}} v = \rho\). Hence, the desired conclusion follows by applying Theorem \eqref{th:main}, since \(\sup_{\mathbb{R}^N_+} u = \sup_{\mathbb{R}^{N-1}} v = \rho\) and \(f(\rho) = 0\). Notice that $\frac{\partial u}{\partial x_N} > 0$ on $\partial{\mathbb{R}^N_+}$ follows from Hopf's lemma which can be applied due to $f(0)\ge0$.
	\end{proof}
	
	\section*{Acknowledgments}
	This research is funded by Vietnam National Foundation for Science and Technology Development (NAFOSTED) under grant number 101.02-2023.35.
	
	\bibliographystyle{abbrvurl}
	\bibliography{../../../references}
	
\end{document}